\documentclass{article}
\usepackage{graphicx} % Required for inserting images

 \usepackage{amsthm,amsmath,amssymb, pgf, tikz, pgfpict2e, soul}
 \usepackage{tocloft}
 \usepackage{geometry}
\usetikzlibrary{decorations.pathmorphing}

\theoremstyle{plain}
\newtheorem{Thm}{Theorem}[section]
\newtheorem{Cor}[Thm]{Corollary}

\newtheorem{lemma}[Thm]{Lemma}
\newtheorem{Prop}[Thm]{Proposition}
\newtheorem{Def}[Thm]{Definition}

\newtheorem{example}[Thm]{Example}

\title{Class-uniformly resolvable designs with all but one block having size two}
\author{ Karen Cordova\\
\small Department of Mathematics and Statistics\\
\small Wellesley College\\
\small Wellesley MA 02481\\
\small\tt kc110@wellesley.edu\
\and
Alexander J. Diesl\\
\small Department of Mathematics and Statistics\\
\small Wellesley College\\
\small Wellesley MA 02481\\
\small\tt adiesl@wellesley.edu\
\and
Micaela Roth\\
\small Department of Mathematics and Statistics\\
\small Wellesley College\\
\small Wellesley MA 02481\\
\small\tt mr116@wellesley.edu\
\and
Ann N. Trenk\\
\small Department of Mathematics and Statistics\\
\small Wellesley College\\
\small Wellesley MA 02481\\
\small\tt atrenk@wellesley.edu\ }
 
\date{\today}

\begin{document}

\maketitle

\begin{abstract}

A Class-Uniformly  Resolvable Design (CURD) is a resolvable design in which each parallel class has the same block structure.  We study CURDS in which each parallel class contains one block of size $m$ and the remaining blocks have size $2$, for  $m \ge 3$.  In addition to establishing necessary conditions for such a CURD to exist, 
we present two general constructions.  The first  transforms a particular type of cyclic design with block size $k$ into a CURD  with partition $m^12^{\frac{n-m}{2}}$ where $m = 2k$.  This construction  is used to generate CURDS with 26 varieties  (where $m=6$) and  with 82 varieties (where $m=10$). The second constructs  a CURD with partition $m^12^{\frac{n-m}{2}}$ for every value of $m$  that is the power of an odd prime.
 \end{abstract}
 
 
\section{Introduction}

\label{sec-intro}

 A teacher has a class of students and assigns the students to groups for a collaborative  activity.  The groups change each  week  but the group sizes do not change.  Is it possible to make group assignments so that after some number of weeks, each pair of students has been in a group together exactly once?

 This question can be stated using the terminology of  design theory where the elements to be partitioned are called \emph{varieties} and the sets of the partition are called \emph{blocks}.  In particular,  our question is modeled by a   Class-Uniformly Resolvable Design, originally defined in \cite{LaReVa91}.   Let $X$ be a set of varieties, $K$ a set of positive integers (the possible block sizes), and $\lambda$ a positive integer.  A \emph{block} is a subset of $X$, i.e., a set of varieties.  A \emph{pairwise balanced } $(n,K,\lambda)$-\emph{design} is a pair $(X,\cal{B})$ where $X$ consists of $n$ varieties   and $\cal B$ is a collection of blocks so that the size of each block is in $K$ and each pair of varieties in $X$ appears together in a block exactly $\lambda$ times.

 In our context, $X$ is the set of students in the class, $K$ is the set of possible group sizes, and $\cal B$ is the set of all groups taken over all weeks.   In keeping our student assignment problem in mind, we often refer to the elements of $X$ as \emph{students} (rather than varieties). There are additional aspects of our problem that are not captured by pairwise balanced designs.  The collection $\cal B$ of blocks must itself be partitioned into sets, where each set represents the group assignments for a particular week.  Furthermore, the group sizes must be the same each week.  Such a design is called a \emph{Class-Uniformly Resolvable Design} which we define formally below.

 \begin{table}
 \centering
 \begin{tabular} {|c |c | c | c | c|} \hline
  & Block $1$ &  Block $2$ &  Block $3$ &  Block $4$ \\ \hline
 Week 1 & $\{0,1,2\}$ & $\{3,6\}$ &  $\{4,7\}$ & $\{5,8\}$ \\ \hline
 Week 2 & $\{3,4,5\}$ & $\{0,6\}$ &  $\{1,7\}$ & $\{2,8\}$ \\ \hline
 Week 3 & $\{6,7,8\}$ & $\{0,3\}$ &  $\{1,4\}$ & $\{2,5\}$ \\ \hline
 Week 4 & $\{2,4,6\}$ & $\{0,8\}$ &  $\{1,5\}$ & $\{3,7\}$ \\ \hline
 Week 5 & $\{0,5,7\}$ & $\{1,6\}$ &  $\{2,3\}$ & $\{4,8\}$ \\ \hline
 Week 6 & $\{1,3,8\}$ & $\{0,4\}$ &  $\{2,7\}$ & $\{5,6\}$ \\ \hline
 \end{tabular}
 \caption{A $3^12^3$-CURD with $n=9$ and $\lambda = 1$ where each row lists the blocks in a given week.}
 \label{nine-table} 
 \end{table}
 
 \begin{Def}
 \rm{Let $(X,\cal B)$ be a pairwise balanced $(n,K,\lambda)$-design.  It is \emph{resolvable} if $\cal B$ can be partitioned into sets (called \emph{parallel classes}) so that each parallel class contains each member of $X$ exactly once.  Furthermore, it is a \emph{Class-Uniformly Resolvable Design} (CURD) if in addition, each parallel class has the same block sizes.  }
 \end{Def}

We illustrate these concepts with an example in which there are six parallel classes, one for each week.  Let $n=9$, $X = \{0,1,2, \ldots, 8\}$, $K = \{2,3\}$, and $\lambda = 1$.  Table~\ref{nine-table} shows the set $\cal B$ consisting of the $24$ blocks. They are listed in six rows, each representing a week of group assignments.  One can check that each element of $X$ appears once in each week, each pair of distinct elements from $X$ appears in exactly one block together, and in each week there is one block of size $3$ and three blocks of size $2$.

For distinct integers $k_1, k_2, \ldots, k_{\ell}$, we say that a CURD has partition $k_1^{p_1}k_2^{p_2} \cdots k_{\ell}^{p_{\ell}}$ if $K = \{k_1, k_2, \ldots, k_{\ell} \}$ and each parallel class has exactly $p_i$ blocks of size $k_i$ for $1 \le i \le \ell$.  
Just as we use the terms \emph{students} and \emph{varieties} interchangeably, we also refer to parallel classes as \emph{weeks}.
  Table~\ref{nine-table} provides an example of  a CURD with partition $3^1 2^3$, and Table~\ref{ten-table-2} provides an example of  a CURD with partition $4^1 2^3$.

\begin{example}
{\rm If $X = \{0,1,2,\ldots, n-1\}$ and $\cal B$ consists of the single block $\{0,1,2,\ldots, n-1\}$, then the pair $(X,\cal B)$ is a CURD we call the \emph{trivial CURD} on $n$ varieties.}
\label{trivial-example}
\end{example}

Much of the work constructing CURDS (e.g., \cite{DaSt01, DaSt04,LaReVa91, WeVa96}) focuses on the instances where $K \subseteq \{2,3\}$ and $\lambda = 1$, that is, blocks have size $2$ or $3$ and each pair of varieties appears together exactly once.    In this paper, we consider CURDS with the partition $m^12^{p_2}$, that is, each parallel class contains a block of size $m$, where $m \ge 3$, and the remaining blocks have size $2$.  Unless indicated otherwise, we 
restrict our attention to the case   $\lambda = 1$.   In our student assignment problem, the partition $m^12^{p_2}$ models the scenario in which each week the teacher works with one large group (containing $m$ students) and the remaining students work  in pairs. If the number of students in the class is $n$, then we know $p_2 = \frac{n-m}{2}$ and we can express our partition as having the form $m^1 2^{(\frac{n-m}{2})}$.  We refer to such a partition as a  $m^1 2^{(\frac{n-m}{2})}$-CURD
and use the terms $m$-block and $2$-block to refer to blocks of size $m$ and $2$ respectively.

 \section{Necessary Conditions and Consequences}
 \label{sec-necessary-conditions}

 In addition to our scenario of dividing  $n$ students into groups, a 
   CURD with $\lambda = 1$  can be represented by partitioning  the edges of the complete graph $K_n$ into color classes.  The vertices of $K_n$ are the varieties, so the edges are pairs of varieties and partitioning the edges ensures that each pair of varieties appears together exactly once.   The color classes correspond to the weeks.
   For the $2^33^1$-CURD in Table~\ref{nine-table},
 the first three weeks of blocks are illustrated in Figure~\ref{fig-nine-partition}
 and the blocks in each week consist of the edges in a  complete graph $K_3$ together  with the edges in three copies of $K_2$.  More generally, for a CURD with partition  $k_1^{p_1}k_2^{p_2} \cdots k_{\ell}^{p_{\ell}}$, each parallel class consists of the edges of $p_i$ copies of the complete graph with $k_i$ vertices for $1 \le i \le \ell$.  
 
 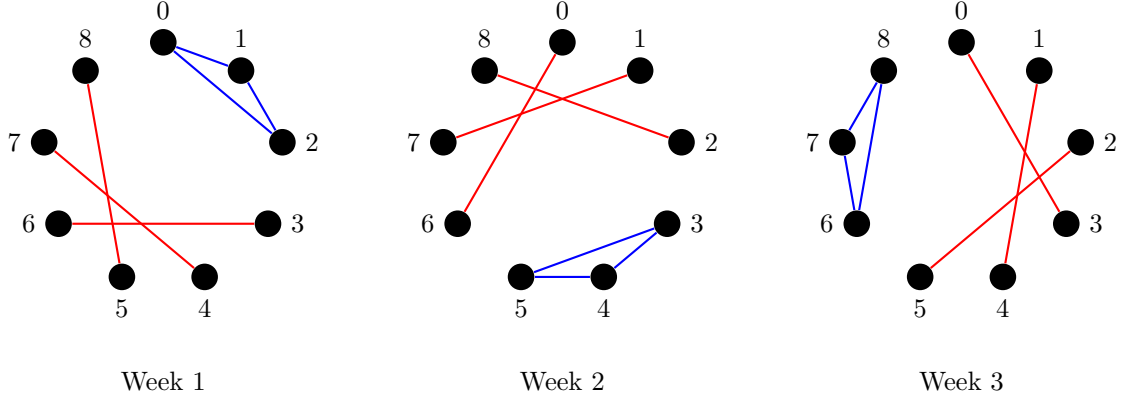
\begin{figure}
\newgeometry{left=0cm, right=0cm}
%\begin{tikzpicture}[scale=1.8]
\begin{tikzpicture}[scale=1.6]
\tikzstyle{vertex}=[circle,fill=black!25,minimum size=10pt,inner sep=3pt]
\begin{scope}[shift={(0,0)}]
\node[vertex,black][label=above:$0$] (0) at (0.0000, 1.0000){};
\node[vertex,black][label=above:$1$] (1) at (0.6428, 0.7660){};
\node[vertex,black][label=right:$2$] (2) at (0.9848, 0.1736){};
\node[vertex,black][label=right:$3$] (3) at (0.8660, -0.5000){};
\node[vertex,black][label=below:$4$] (4) at (0.3420, -0.9397){};
\node[vertex,black][label=below:$5$] (5) at (-0.3420, -0.9397){};
\node[vertex,black][label=left:$6$] (6) at (-0.8660, -0.5000){};
\node[vertex,black][label=left:$7$] (7) at (-0.9848, 0.1736){};
\node[vertex,black][label=above:$8$] (8) at (-0.6428, 0.7660){};

\draw[blue, thick] (0) -- (1) -- (2) -- (0) -- cycle;
\draw[red, thick] (3) -- (6) -- cycle;
\draw[red, thick] (4) -- (7) -- cycle;
\draw[red, thick] (5) -- (8) -- cycle;

\node at (0,-1.8) {Week 1};
\end{scope}

\begin{scope}[shift={(3.3,0)}]
\node[vertex,black][label=above:$0$] (0) at (0.0000, 1.0000){};
\node[vertex,black][label=above:$1$] (1) at (0.6428, 0.7660){};
\node[vertex,black][label=right:$2$] (2) at (0.9848, 0.1736){};
\node[vertex,black][label=right:$3$] (3) at (0.8660, -0.5000){};
\node[vertex,black][label=below:$4$] (4) at (0.3420, -0.9397){};
\node[vertex,black][label=below:$5$] (5) at (-0.3420, -0.9397){};
\node[vertex,black][label=left:$6$] (6) at (-0.8660, -0.5000){};
\node[vertex,black][label=left:$7$] (7) at (-0.9848, 0.1736){};
\node[vertex,black][label=above:$8$] (8) at (-0.6428, 0.7660){};

\draw[blue, thick] (3) -- (4) -- (5) -- (3) -- cycle;
\draw[red, thick] (0) -- (6) -- cycle;
\draw[red, thick] (1) -- (7) -- cycle;
\draw[red, thick] (2) -- (8) -- cycle;

\node at (0,-1.8) {Week 2};
\end{scope}

\begin{scope}[shift={(6.6,0)}]
\node[vertex,black][label=above:$0$] (0) at (0.0000, 1.0000){};
\node[vertex,black][label=above:$1$] (1) at (0.6428, 0.7660){};
\node[vertex,black][label=right:$2$] (2) at (0.9848, 0.1736){};
\node[vertex,black][label=right:$3$] (3) at (0.8660, -0.5000){};
\node[vertex,black][label=below:$4$] (4) at (0.3420, -0.9397){};
\node[vertex,black][label=below:$5$] (5) at (-0.3420, -0.9397){};
\node[vertex,black][label=left:$6$] (6) at (-0.8660, -0.5000){};
\node[vertex,black,black][label=left:$7$] (7) at (-0.9848, 0.1736){};
\node[vertex,black][label=above:$8$] (8) at (-0.6428, 0.7660){};

\draw[blue, thick] (6) -- (7) -- (8) -- (6) -- cycle;
\draw[red, thick] (0) -- (3) -- cycle;
\draw[red, thick] (1) -- (4) -- cycle;
\draw[red, thick] (2) -- (5) -- cycle;

\node at (0,-1.8) {Week 3};
\end{scope}

\end{tikzpicture}
\restoregeometry

\caption{Illustration of first three parallel classes for the CURD in Table~\ref{nine-table} }
\label{fig-nine-partition}
\end{figure}

 \subsection{Three Necessary Conditions}

We next present some necessary conditions for the existence of CURD with partition $m^1 2^{(\frac{n-m}{2})}$ and $\lambda = 1$, although several of the conditions can be  extended to any integer $\lambda \ge 1$.  More general necessary conditions for the existence of  CURDS can be found in \cite{DaSt01}.
 Our first condition determines the number of weeks  required in the $m^1 2^{(\frac{n-m}{2})}$-student assignment problem.  
 
 \begin{Thm}
If $(X,\cal B)$ is a CURD with partition $m^1 2^{(\frac{n-m}{2})}$ and $w$ is the number of parallel classes, then  $ w = \frac{n(n-1)}{m^2 -2m + n}$. 
\label{nec-cond-1}
 \end{Thm}
 
 \begin{proof}
 There are a total of $\binom{n}{2}$  edges in the complete graph $K_n$, thus there are a total of $\frac{n(n-1)}{2}$ edges to partition.  In each parallel class there are $\binom{m}{2}$ edges in $K_m$ and one edge in each  of the $\frac{n-m}{2}$ copies of $K_2$ for a total of $\frac{m(m-1)}{2} + \frac{n-m}{2} = \frac{m^2 -2m + n}{2}$.  Thus $w = \frac{n(n-1)}{m^2 -2m + n}$. 
 \end{proof}
 
Our second necessary condition shows that the number of weeks in which a student is in an $m$-block is the same for each student.

  \begin{Thm}
If $m \ge 3$ and  $(X,\cal B)$ is a CURD with partition $m^1 2^{(\frac{n-m}{2})}$,   and $w$ is the number of parallel classes, then each student is in exactly $t$ blocks of size $m$ where  $t=\frac{n-w-1}{m-2}$.
\label{nec-cond-2}
 \end{Thm}
 
 \begin{proof}
 Let $t_i$ be the number of $m$-blocks in our CURD  that contain student $i$.  Thus student $i$ is in an $m$-block  for $t_i$ weeks and a $2$-block   for the remaining $w-t_i$ weeks.  There are $n-1$ edges incident to the vertex represented by student $i$ in the complete graph $K_n$.  Each week in which student $i$ is in an $m$-block  accounts for $m-1$ of these edges, and each week in which they are in a $2$-block  accounts for one of these edges.  Thus $n-1 = t_i(m-1) + (w-t_i)$ and solving for $t_i$ we obtain $t_i = \frac{n-1-w}{m-2}$.  We know that $w$ is determined by $n$ and $m$ by Theorem~\ref{nec-cond-1}, so the value of $t_i$ depends only on $n$ and $m$ and is therefore independent of $i$.
 \end{proof}

 \begin{Thm}
Let $(X,\cal B)$ be a CURD with partition $m^1 2^{(\frac{n-m}{2})}$,    where $w$ is the number of parallel classes, $m \ge 3$,  and each student is in exactly $t$   blocks of size $m$.  Then $mw=tn$.
\label{nec-cond-3}
 \end{Thm}
 
 \begin{proof}
 Let $A$ be the $w \times n$ incidence matrix for which $A_{ij} = 1$ if student $j$ is in an $m$-block during week $i$, and $A_{ij} = 0$  otherwise.  We calculate the sum of all entries in $A$ in two ways and set them equal.  In week $i$ there will be $m$ students in the $m$-block so the sum of the entries in row $i$ is $m$ for each $i$.  There are $w$ rows in matrix $A$, so the sum of the entries is $mw$.  By Theorem~\ref{nec-cond-2}, each student $j$ is an an $m$-block exactly $t$ times, so the sum of the entries in column $j$ is $t$.  There are $n$ students, so $n$ columns in matrix $A$, so the sum of the entries is $tn$.  Hence $mw = tn$ as desired.
 \end{proof}

 \subsection{Consequences of our necessary conditions}
 
 Our first consequence is that the trivial CURD is the only   $m^1 2^{(\frac{n-m}{2})}$ CURD with $t=1$. \\

 \begin{Thm} 
For $m \ge 3$, there is a unique $m^1 2^{(\frac{n-m}{2})}$-CURD (up to isomorphism) in which each student is in exactly one $m$-block.   Moreover, it is
the trivial CURD from Example~\ref{trivial-example}  and it has $n=m$.
 \label{trivial-unique-thm}
 \end{Thm}
 
 \begin{proof}
 Let $(X,\cal B)$ be an $m^1 2^{(\frac{n-m}{2})}$-CURD  and let $w$ be the number of parallel classes in $(X,\cal B)$. 
Since each  student is in exactly one $m$-block, we know $t=1$.  
   By Theorem~\ref{nec-cond-3} we know that $mw = n$ so $w = \frac{n}{m}$.  Now we apply Theorem~\ref{nec-cond-2} with $t=1$ to obtain $\frac{n-w-1}{m-2} = 1$ or equivalently, $n-w-1 = m-2$.  Substituting $w = \frac{n}{m}$ yields $n - \frac{n}{m} = m-1$ which simplifies to $\frac{n(m-1)}{m} = m-1$ and thus $n=m$.  Hence $w=1$ and there is only one parallel class in $(X,\cal B)$, which means there is one block in $\cal B$ and it has size $n$, so we get the trivial CURD on $n$ varieties.
 \end{proof}
 
 Theorem~\ref{trivial-unique-thm} handles the case in which $t=1$, so we next focus on cases in which $t \ge 2$.

  \begin{Thm}
  For $m \ge 3$ and $t \ge 2$, 
let $(X,\cal B)$ be a CURD with partition $m^1 2^{(\frac{n-m}{2})}$    and  for which each student is in exactly $t$  blocks of size $m$.  Then the following  hold:  
\begin{enumerate}
\item[{\rm (i)}]  $n = \frac{m(mt -2t + 1)}{m-t}$ 
\item[{\rm (ii)}] $m \le t(t^2-2t+2)$
\item[{\rm (iii)}]  $m \ge t+1$
\item[{\rm (iv)}]   $n \le t^4 - t^3 + t^2 - t + 1$

\end{enumerate}

\label{m-bound-thm}
 \end{Thm}
 
 \begin{proof}
 Let $w$ be the number of parallel classes, so by Theorem~\ref{nec-cond-3} we know $t = \frac{mw}{n}$.  Combining this with Theorem~\ref{nec-cond-1}  we get  $t = \frac{m(n-1)}{m^2-2m+n}$ or equivalently $t(m^2 -2m + n)= mn-m$.  
 Solving for $n$ yields $n = \frac{m^2t -2mt + m}{m-t} =  \frac{m(mt -2t + 1)}{m-t}$, proving (i). 
 
 Now using long division we obtain $n = mt + (t-1)^2 + \frac{t(t-1)^2}{m-t}$.  Since $n, m$, and $t$ are integers we conclude that $t(t-1)^2$ is divisible by $m-t$.  We  are given $t \ge 2$ so $t(t-1)^2 > 0$ and thus $m-t \le t(t-1)^2$ which simplifies to $m \le t(t^2-2t + 2)$.   This proves (ii).

 To prove (iii) we start with $t(m^2 -2m + n)= mn-m$ from above.  For a contradiction, assume $m \le t$ so 
 $m(m^2 -2m + n) \le  mn-m$.  Dividing by $m$ and then subtracting $n$ from both sides yields  $m(m-2)  \le  -1$.  This is a contradiction since $m \ge 3$.
 
 Finally, to prove (iv) we start with the expression  $n = mt + (t-1)^2 + \frac{t(t-1)^2}{m-t} $  from above   and substitute $m \le t(t^2-2t + 2)$ from (ii)  to get 
  $n  \le  t^2(t^2-2t + 2) + (t-1)^2 + \frac{t(t-1)^2}{m-t}$.  Now using $m -t \ge 1$ from (iii)  yields
  $ n \le   t^2(t^2-2t + 2) + (t-1)^2 + t(t-1)^2 =   t^4 - t^3 + t^2 - t + 1.$
 \end{proof}
 
 \subsection{Limitations on $n$ for the existence of a 
  CURD with partition $m^1 2^{(\frac{n-m}{2})}$}

For a fixed integer $m$, our necessary conditions and their consequences limit the possible values of $n$ for which a  $m^1 2^{(\frac{n-m}{2})}$ CURD exists.  We first show that there are at most $m-1$ values of $n$ that yield a solution to the $m^1 2^{(\frac{n-m}{2})}$-student assignment problem.  

\begin{Thm}
    If $m$ is an integer with $m \ge 3$ then there are at most $m-1$ values of $n$ for which there exists a CURD with partition  $m^1 2^{(\frac{n-m}{2})}$.
    \label{max-number-thm}
\end{Thm}

\begin{proof}

    Suppose there exists a  CURD with partition $m^1 2^{(\frac{n-m}{2})}$.  By Theorem~\ref{nec-cond-2}, each student is in exactly $t$ blocks of size $m$ and by Theorem~\ref{m-bound-thm}(iii) we know $t \le m-1$.  
For any $m\ge 3$, fixing a value of $t$ determines the value of $n$  by   Theorem~\ref{m-bound-thm} (i).   Thus there are at most $m-1$ possible values for $n$.
\end{proof}

In general,  some of the  values of $t$ between $1$ and $m-1$  may fail to satisfy conditions (ii) and (iv) of Theorem~\ref{m-bound-thm}.
 However, for $m=3$ there are precisely two values of $n$ for which a CURD with partition $m^1 2^{(\frac{n-m}{2})}$ exists.

 \begin{Thm}
 There exists a CURD with partition $3^1 2^{(\frac{n-3}{2})}$ if and only if $n=3$ or $n=9$.
 \label{m-equals-3-thm}
 \end{Thm}
 
 \begin{proof}
 When $n=3$ we have the trivial CURD  with partition $3^1$ given in Theorem~\ref{trivial-unique-thm}.  When $n=9$, we have the CURD with partition $3^1 2^3$ given in  Table~\ref{nine-table}. 
By Theorem~\ref{max-number-thm}, there are no additional values of $n$ for which a  CURD  with partition $3^1 2^{(\frac{n-3}{2})}$ exists.
  \end{proof}

The next two theorems consider the cases $m=4$ and $m=5$.

 \begin{Thm}
 If $(X,\cal B)$ is  a CURD with partition $4^1 2^{(\frac{n-4}{2})}$ then $n=4, \,10$ or $28$.
 \label{m-equals-4-thm}
 \end{Thm}
 
 \begin{proof}
Suppose $(X,\cal B)$ is a CURD with partition $4^1 2^{(\frac{n-4}{2})}$.  By Theorem~\ref{nec-cond-2},
  each variety is in exactly $t$ blocks of size $4$.   
 
  By Theorem~\ref{m-bound-thm}(iii) we know $1 \le t \le 3$.   Now we apply Theorem~\ref{m-bound-thm}(i):  when    $t=1$  we get $n=4$, when $t=2$ we get $n=10$, and when $t=3$ we get $n=28$.
  \end{proof}

The trivial CURD has $n=4$, and 
Table~\ref{ten-table-2}  shows a  $4^1 2^3$-CURD which has $n=10$.  We do not know if a $4^1 2^{12}$-CURD exists.

 \begin{Thm}
 If $(X,\cal B)$ is  a CURD with partition $5^1 2^{(\frac{n-5}{2})}$ then $n=5, \, 25$ or $65$.
 \label{m-equals-5-thm}
 \end{Thm}
 
 \begin{proof}
Suppose $(X,\cal B)$ is a CURD with partition $5^1 2^{(\frac{n-4}{2})}$, thus $m=5$.  By Theorem~\ref{nec-cond-2},
  each variety is in exactly $t$ blocks of size $5$.     By Theorem~\ref{m-bound-thm}(iii) we know $  t \le 4$.  The case $t=1$ is handled in Theorem~\ref{trivial-unique-thm}; in this case, $(X,\cal B)$ is the trivial CURD with $n=m$, so $n=5$.    When $t=2$ we use 
  Theorem~\ref{m-bound-thm}(ii) to conclude $m \le 4$, a contradiction, so $t=3$ or $t=4$.  
  Now applying Theorem~\ref{m-bound-thm}(i),  when    $t=3$  we get $n=25$, and  when $t=4$ we get $n=65$.
    \end{proof}

We will see in Theorem~\ref{primecurd} that a 
 $5^12^{10}$-CURD exists.  We do not know if a $5^12^{30}$-CURD exists.  In the remaining sections we provide  additional constructions and construction techniques.

 \section{Constructing a CURD from a cyclic design}

Table~\ref{ten-table-2} provides an example of a $4^1 2^3$-CURD.  In this section, we will explore a way to generalize this example, developing a method that starts with a cyclic  balanced incomplete block design that is generated by two base blocks, and  in certain circumstances transforms it into  a CURD with partition $m^1 2^{(\frac{n-m}{2})}$. 

We begin by recalling some fundamental definitions.
For additional background on combinatorial designs, see \cite{CoDi07}.

A \emph{design}  is a pair $(X, \mathcal B) $ where $X$ is a finite set (of varieties) and $\mathcal B$ is a multiset consisting of non-empty subsets of $X$ (the blocks).  A design 
 $(X, \mathcal B) $  in which $|X| = v$ and $|\mathcal B| = b$ is a  $(b,v,r,k,\lambda)$-\emph{balanced incomplete block design} (BIBD) if there exist integers $r,k,\lambda$ so that each variety is in exactly $r$ blocks ($r$-regular), each block contains exactly $k$ varieties ($k$-uniform), each pair of varieties appears together in exactly $\lambda$ blocks ($\lambda$-balanced), and $k < v$ (incomplete).  
If $v$ is a positive integer, we define the \emph{ multiset of modulo $v$ differences }  of $\{x_1, x_2, \ldots, x_k\}$ to be the   multiset  $\{ \pm(x_i - x_j) \pmod v : 1 < i < j < k \}$, where each element is expressed as a non-negative residue modulo $v$. For example, the  multiset of modulo $13$ differences of $\{2,5,6\}$ is $\{1,3,4,9,10,12\}$ and the  multiset of modulo $13$ differences of $\{1,2,3\}$ is $\{1,1,2,11, 12,12\}$.

A block design with variety set $\{0,1,2,\ldots, v-1\}$ is \emph{cyclic} if it can be generated from a set of base blocks by successively adding $j$ (in modulo $v$) to each element of the base blocks.   A classic result in design theory 
tells us that a set $\cal B$ of base blocks generates a cyclic $\lambda$-balanced design if and only if each  element in $\{1,2,3, \ldots, v-1\}$ occurs exactly $\lambda$ times  in the union of the multisets of modulo $v$ differences taken over all base blocks in $\mathcal B$. 
We illustrate this in the following example.

\begin{example}
{\rm
Let $B = \{1,2\}$ and $Y = \{0,3\}$.  The multiset of modulo $5$ differences of $Y$ is $\{2,3\}$ and   that of $B$ is $\{1,4\}$.  Since each element of $\{1,2,3,4\}$ appears exactly once in the union of these difference multisets, the set  $\{B,Y\}$ generates a cyclic block design with $\lambda = 1$.
}
\label{cyclic-ex1}
\end{example}

\begin{figure}
\centering
 \begin{tikzpicture}[scale=0.65]

\draw(6,7) circle [radius = 3.5];
\draw (4,8) circle [radius=3pt];
 \draw (6,5) circle [radius=3pt];
  \draw (6,9) circle [radius=3pt];
    \draw (8,6) circle [radius=3pt];  
    \draw (8,8) circle [radius=3pt];
    \draw(16,7) circle [radius = 3.5];
    \draw (18,6) circle [radius=3pt];
 \draw (16,5) circle [radius=3pt];
  \draw (14,9) circle [radius=3pt];
    \draw (14,6) circle [radius=3pt];  
    \draw (17,8) circle [radius=3pt];

\tikzstyle{vertex}=[circle,fill=black!25,minimum size=10pt,inner sep=3pt]

\node[vertex, black][label=above:$1$] (1) at (8,8){};
\node[vertex, black][label=above:$0'$] (0') at (17,8){};
\node[vertex, black][label=below:$3'$] (3') at (14,6){};
\node[vertex, black][label=below:$2$] (2) at (8,6){};

\draw[black, thick]  (1) -- (0') -- (3') -- (2) -- (1);
\draw[black, thick]  (1) -- (3') ;
\draw[black, thick]  (0')  -- (2);

\node[vertex, red][label=above:$0$] (0) at (6,9) {}; 
\node[vertex, red][label=below:$3$] (3) at (6,5) {}; 
\draw[red,thick,decorate, decoration={snake, amplitude=.5mm, segment length=5mm}] (0) -- (3);

\node[vertex, red][label=right:$1'$] (1') at (18,6){};
\node[vertex, red][label=below:$2'$] (2') at (16,5){};
\draw[red,thick,decorate, decoration={snake, amplitude=.5mm, segment length=5mm}] (1') -- (2');

\node[vertex, blue][label=below:$4$] (4) at (4,8){};
\node[vertex, blue][label=right:$4'$] (4') at (14,9){};
        
 \draw[blue,thick,dashed] (4) .. controls (5,13) and (11,12) .. (4');
\end{tikzpicture}

\caption{Illustration of  the week 1   parallel class for the CURD in Table~\ref{ten-table-2} }
\label{fig-base-blocks-10}
 \end{figure}
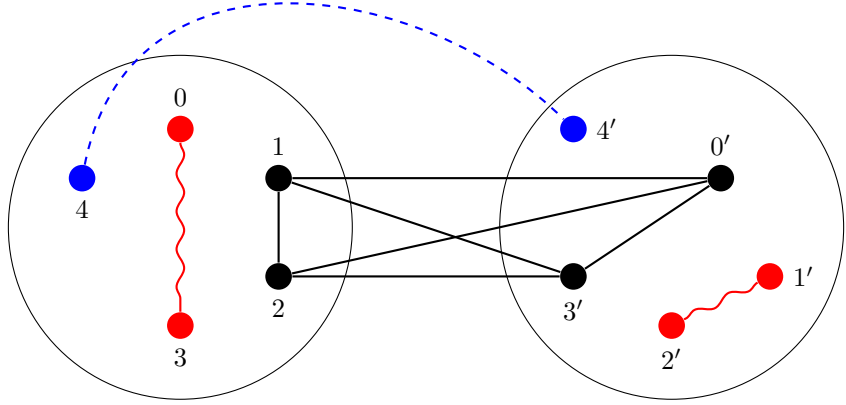

In Table~\ref{ten-table-2}, we use the cyclic block design  from Example~\ref{cyclic-ex1} to construct a $4^12^3$-CURD.  The construction and verification is provided in Theorem~\ref{curd-from-design-thm} and Proposition~\ref{ten-verif-prop}.
The CURD shown in Table~\ref{ten-table-2}  comes from utilizing the cyclic block design of Example~\ref{cyclic-ex1} on two different sets of  varieties, ${S} = \{0,1,2,3,4\}$ and ${S'} = \{0',1',2',3',4'\}$.     The blocks in week $1$ of the table are the base blocks for the $4^12^3$-CURD, and they are  also shown in Figure~\ref{fig-base-blocks-10}.   Subsequent rows of the table are constructed  by adding $1$ to each entry as described formally below.
The base block of size 4 in Table~\ref{fig-base-blocks-26} is obtained by combining the elements of ${S} $ in $B$ with the elements of ${S}' $ that naturally correspond to elements of $Y$.  The remaining base blocks come in three categories:  those with both elements in ${S} $, those with both elements in ${S}' $,  and those with one element in each of these sets.

   \begin{table}
 \centering
 \begin{tabular} {|c |c | c | c | c|} \hline
  & Block $1$ &  Block $2$ &  Block $3$ &  Block $4$ \\ \hline
 Week 1 & $\{1,2,0',3'\} $ & $\{0,3\}$ &  $\{1',2'\}$ & $\{4,4'\}$ \\ \hline
 Week 2 &  $\{2,3,1',4'\} $ & $\{1,4\}$ &  $\{2',3'\}$ & $\{0,0'\}$ \\ \hline
 Week 3 & $\{3,4,2',0'\} $ & $\{2,0\}$ &  $\{3',4'\}$ & $\{1,1'\}$ \\ \hline
 Week 4 & $\{4,0,3',1'\} $ & $\{3,1\}$ &  $\{4',0'\}$ & $\{2,2'\}$ \\ \hline
 Week 5 & $\{0,1,4',2'\} $ & $\{4,2\}$ &  $\{0',1'\}$ & $\{3,3'\}$ \\ \hline
 \end{tabular}
 \caption{A  $4^12^3$-CURD with $n=10$ and $\lambda = 1$ where each row lists the blocks in a given week.}
 \label{ten-table-2} 
 \end{table}
 
  \begin{figure}[hbt!]
\centering

\begin{tikzpicture}

\path[use as bounding box] (-3.3,-4.5) rectangle (10.3,3.3);

\tikzstyle{vertex}=[circle,fill=black!25,minimum size=10pt,inner sep=3pt]

\node[vertex,blue][label=above:$0$] (0) at (0.0000,2.3000){};
\node[vertex,black][label=above right:$1$] (1) at (1.0688,2.0367){};
\node[vertex,red][label=right:$2$] (2) at (1.8929,1.3066){};
\node[vertex,black][label=right:$3$] (3) at (2.2832,0.2771){};
\node[vertex,red][label=right:$4$] (4) at (2.1505,-0.8156){};
\node[vertex,red][label=below right:$5$] (5) at (1.5251,-1.7216){};
\node[vertex,red][label=below:$6$] (6) at (0.5504,-2.2331){};
\node[vertex,blue][label=below:$7$] (7) at (-0.5504,-2.2331){};
\node[vertex,blue][label=below left:$8$] (8) at (-1.5251,-1.7216){};
\node[vertex,black][label=left:$9$] (9) at (-2.1505,-0.8156){};
\node[vertex,red][label=left:$10$] (10) at (-2.2832,0.2771){};
\node[vertex,blue][label=left:$11$] (11) at (-1.8929,1.3066){};
\node[vertex,red][label=above left:$12$] (12) at (-1.0688,2.0367){};

\draw[red,thick,decorate,decoration={snake,amplitude=1.15mm,segment length=11.5mm}] (2)--(12)--cycle;
\draw[thick] (1)--(3)--(9)--(1)--cycle;
\draw[red,thick,decorate,decoration={snake,amplitude=1.15mm,segment length=11.5mm}] (6)--(10)--cycle;
\draw[red,thick,decorate,decoration={snake,amplitude=1.15mm,segment length=11.5mm}] (4)--(5)--cycle;

\node[vertex,blue][label=above:$0'$] (0') at (6.9000,2.3000){};
\node[vertex,blue][label=above right:$1'$] (1') at (7.9688,2.0367){};
\node[vertex,black][label=right:$2'$] (2') at (8.7929,1.3066){};
\node[vertex,blue][label=right:$3'$] (3') at (9.1832,0.2771){};
\node[vertex,red][label=right:$4'$] (4') at (9.0505,-0.8156){};
\node[vertex,black][label=below right:$5'$] (5') at (8.4251,-1.7216){};
\node[vertex,black][label=below:$6'$] (6') at (7.4504,-2.2331){};
\node[vertex,red][label=below:$7'$] (7') at (6.3496,-2.2331){};
\node[vertex,red][label=below left:$8'$] (8') at (5.3759,-1.7216){};
\node[vertex,blue][label=left:$9'$] (9') at (4.7495,-0.8156){};
\node[vertex,red][label=left:$10'$] (10') at (4.6168,0.2771){};
\node[vertex,red][label=left:$11'$] (11') at (5.0071,1.3066){};
\node[vertex,red][label=above left:$12'$] (12') at (5.8312,2.0367){};

\draw[thick] (2')--(5')--(6')--(2')--cycle;
\draw[red,thick,decorate,decoration={snake,amplitude=1.15mm,segment length=11.5mm}] (11')--(4')--cycle;
\draw[red,thick,decorate,decoration={snake,amplitude=1.15mm,segment length=11.5mm}] (12')--(7')--cycle;
\draw[red,thick,decorate,decoration={snake,amplitude=1.15mm,segment length=11.5mm}] (8')--(10')--cycle;

% bendy stuff
\draw[blue,thick,dashed] (0)..controls (3.45,3.45)..(0');
\draw[blue,thick,dashed] (11)..controls (1.725,-1.84)..(9');
\draw[blue,thick,dashed] (7)..controls (2.3,-4.14) and (4.14,-2.3)..(1');
\draw[blue,thick,dashed] (8)..controls (1.15,-6.9) and (12.65,-3.22)..(3');

\draw (0,0) circle (3.22);
\draw (6.9,0) circle (3.22);

\end{tikzpicture}
\caption{Illustration of first   parallel class for the  $6^1 2^{10}$-CURD in Proposition~\ref{prop-26}.  The edges between vertices in $\{1,3,9\}$ and $\{2',5',6'\}$ are not shown.}
\label{fig-base-blocks-26}
\end{figure}
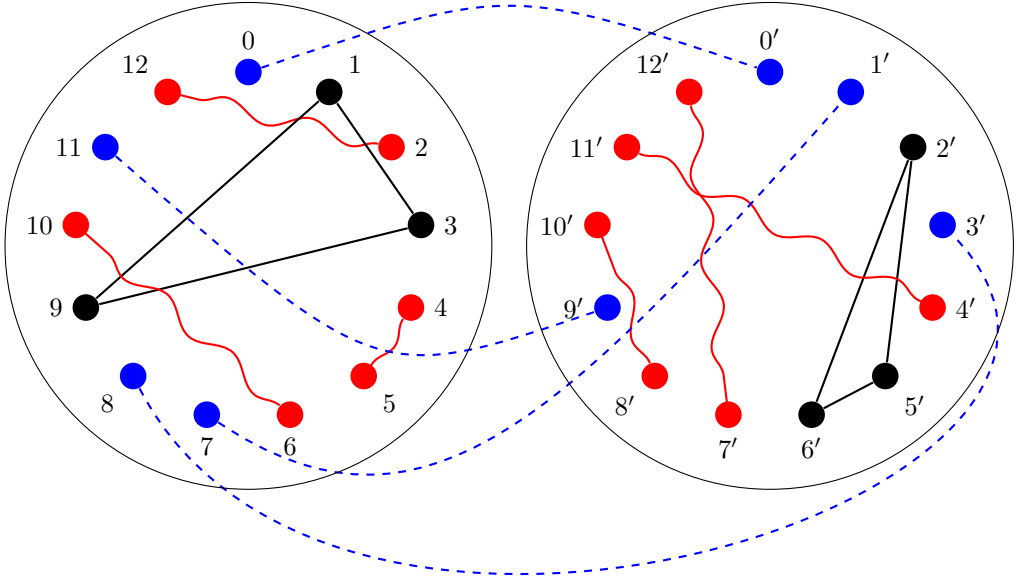

More generally, we start with a cyclic $(b,v,r,k, \lambda)$-balanced incomplete block design  that is generated by two base blocks and under   conditions given in Theorem~\ref{curd-from-design-thm}, we construct a  CURD with partition $m^1 2^{\frac{n-m}{2}}$ where $m = 2k$ and $n = 2v$.  We begin by duplicating the set of varieties $ V$ to sets ${\cal S}$ and ${\cal S'}$ where  ${\cal S} =  \{0, 1, 2, \ldots, v-1\}$ and ${\cal S}' = \{0', 1', 2', \ldots, (v-1)'\}$. 
 There is a natural correspondence between ${\cal S}$ and ${\cal S}'$, namely $\sigma:{\cal S} \to {\cal S'}$ defined by $\sigma(i) = i'$.
For $i,k \in {\cal S}$, define the operation $\oplus$   by $i \oplus k = i+k {\pmod v}$, and $\sigma(i) \oplus k = \sigma(i\oplus k)$, so for example, when $v=5$ we have $4' \oplus 2  = \sigma(4) \oplus 2 = \sigma(4 \oplus 2) = \sigma(1) = 1'$.

We are  now ready to state the main theorem of this section.

\begin{Thm}
Let $(V, \cal B)$ be a cyclic 
$(b,v,r,k, \lambda)$-BIBD with $\lambda = 1$ that is 
generated by base blocks $B$ and $Y$ where $B = \{b_1, b_2, \ldots, b_k\}$, \  $Y = \{y_1, y_2, \ldots, y_k\}$, and ${V} = \{0,1,2,\ldots, v-1\}$. Let $D_1$ be the multiset of modulo $v$ differences of $B$ and let $D_2$ be the multiset of  modulo $v$ differences  of $Y$.   Further, let ${ S} = \{0,1,2, \ldots, v-1\}$ and  ${ S}' = \{0', 1', 2', \ldots, (v-1)'\}$ and define $\sigma: S\to S'$ by $\sigma(i) = i'$.  Suppose that we can partition ${V} - B$ as $A \cup C$ and ${V} - Y$ as $X \cup Z$ so that the following  three conditions hold:
\begin{enumerate}
\item There exists a partition of  the elements of $A$ into $\binom{k}{2}$ unordered pairs $\{a_1,\hat{a}_1\},  \{a_2,\hat{a}_2\},   \ldots,  \{a_{\binom{k}{2}},
\hat{a}_{\binom{k}{2}} \}$ so that $\{\pm(\hat{a}_i - a_i) \pmod v\} = D_2$.

\item There exists a partition of  the elements of $X$ into $\binom{k}{2}$ unordered pairs  $\{x_1,\hat{x}_1\},  \{x_2,\hat{x}_2\},   \ldots,  \{x_{\binom{k}{2}},
\hat{x}_{\binom{k}{2}} \}$ so that $\{\pm(\hat{x}_i - x_i) \pmod v\}= D_1$.

\item There exists a bijection $\phi:C \to Z$ where $C = \{c_1, c_2, \ldots, c_{\ell} \}$, $Z = \{z_1, z_2, \ldots, z_{\ell} \}$, and $\phi(c_i) = z_i$ so that the sets $\{(z_i-c_i) \pmod v: 1 \le i \le \ell\}$ and $\{(y_j-b_i) \pmod v: 1 \le i,j \le k\}$  partition $\{0,1,2, \ldots, v-1\}$.
\end{enumerate}

Then the following  set of base blocks generates a  CURD with partition $m^1 \, 2^{\frac{n-m}{2}}$ under the addition $\oplus$  modulo $v$ where the set of varieties is $S \cup S'$ and $m = 2k$.

\begin{itemize}
\item Size $m$ base block: $\{b_1, b_2, \ldots, b_k\} \cup  \{\sigma(y_1), \sigma(y_2),\ldots, \sigma(y_k)\}$,

\item Size $2$ base blocks within $S$:  $\{ \{a_i,\hat{a}_i\} :  1 \le i \le \binom{k}{2} \}$

\item Size $2$ base blocks within $S'$:  $\{ \{\sigma(x_i), \sigma(\hat{x}_i)\}:  1 \le i \le \binom{k}{2} \}$

\item Size $2$ base blocks between $S$ and $S'$:  $ \{ c_i, \sigma(z_i)\} : 1     \le i \le \ell \}$.

\end{itemize}

\label{curd-from-design-thm}
\end{Thm}

\begin{proof}
 Let $\cal D$ be the design on the variety set  ${\cal S} \cup {\cal S}'$  generated by these four types of base blocks.
We first show that each element of  ${\cal S} \cup {\cal S}'$ appears in exactly one  base block of $\cal D$.
First consider element $j \in {\cal S}$.    Since ${V} $ is partitioned as $A \cup B \cup C$, we know that $j$ is in exactly one of the sets $A,B,C$.    If $j \in A$ then $j$ is in a size $2$ base block within ${\cal S}$, if $j \in B$ then $j$ is in the size $m$ base block, and if  $j \in C$ then $j$ is in a size $2$ base block between  ${\cal S}$ and ${\cal S}'$.  A similar result holds for $\sigma(j) \in {\cal S}'$.    Blocks in subsequent weeks are constructed using the function that sends  $j \to  j \oplus 1$,  and 
$\sigma(j)  \to  \sigma(j \oplus 1)$,
which is one-to-one, so in each subsequent week there are no repeated vertices, and therefore each appears exactly once in each week.

We next check that every pair of elements appears together in a block.  Since $\{B,Y\}$ generate a cyclic BIBD with $\lambda = 1$ we know that the multisets $D_1$ and $D_2$ are actually sets, and that $D_1 \cup D_2 $ forms a partition of $ \{1,2,3, \ldots, v-1\}$.  First consider two elements $u,w \in {\cal S}$.  If the mod $v$ residue of $u-w$ is in $D_1$ then $u$ and $w$ appear together in a block generated by the size $m$ base block, and if it is in $D_2$ then they appear together in a block generated by one of the size $2$ blocks within ${\cal S}$.    Analogously, any two distinct elements of ${\cal S}'$ appear together in a block.

Now consider $u \in {\cal S}$ and 
$\sigma(w) \in {\cal S}'$  for some $w \in {\cal S}$, 
and the modulo $v$ residue of $w-u$.  If this residue appears in the set $\{(y_j-b_i) \pmod v: 1 \le i,j \le k\}$ then $u$ and 
 $\sigma(w)$  appear together in a block generated by the size $m$ base block.  Otherwise,  by condition (3),  this residue appears in  $\{(z_i-c_i) \pmod v: 1 \le i \le \ell\}$ and $u$ and $\sigma(w)$ 
appear together in a block generated by one of the size $2$ base blocks between ${\cal S}$ and ${\cal S}'$.

Finally, we show that no two elements appear together more than once by a counting argument.  First we write $v$ and $\ell$ in terms of $k$.  Since $|D_1| = |D_2| = 2 \binom{k}{2}$ and $D_1 \cup D_2$ is a partition of $\{1,2,3,  \ldots, v-1\}$, we know $v = 4 \binom{k}{2} + 1  = 2k^2 -2k + 1$.    In addition, set $\cal S$ is partitioned as $A \cup B \cup C$ where $|{\cal S}| = v$,  and $|A| =  2 \binom{k}{2}$, and $|B| = k$.  Thus $|C| = \ell  = v- 2\binom{k}{2} - k = k^2 - 2k + 1$.  

Consider the graph $G$ whose vertex set is ${\cal S} \cup {\cal S}'$ and where two vertices are adjacent if they are together in one of the base blocks.   
We compute $|E(G)|$ by considering edges that arise from the four different categories of base blocks.  
 The graph induced in $G$ by the vertices in the size $m$ base block has $\binom{m}{2}$ edges and $\binom{m}{2} = \binom{2k}{2} = 2k^2 - k$.   There are $\binom{k}{2}$ edges resulting from the size $2$ base blocks within $\cal S$ and an additional $\binom{k}{2} $ edges from the size $2$ base blocks within ${\cal S}'$ for a total of $k^2-k$.  The remaining $\ell$ edges come from the size $2$ base blocks with one vertex in $\cal S$ and the other in ${\cal S}'$.    Hence $|E(G)| = (2k^2 - k) + (k^2-k) + (k^2-2k+1) =4k^2 -4k + 1$.   Thus in week 1 (and each subsequent week) there are $4k^2-4k+1$ pairs of elements of ${\cal S} \cup {\cal S}'$  that appear together in a block.  
 
 Since there are $v$ weeks,  at most $v(4k^2-4k+1)$  distinct pairs of elements of ${\cal S} \cup {\cal S}'$  appear together in a block of $\cal D$. The set ${\cal S} \cup {\cal S}'$ contains  $\binom{2v}{2}$  pairs of distinct elements, and we proved above that each pair appears together in a block of $\cal D$.    However,    $\binom{2v}{2} = v(4k^2 - 4k + 1)$,  so each pair appears together \emph{exactly once} in a block of $\cal D$ and our design is pairwise balanced with $\lambda = 1$ as desired. 
\end{proof}

 There are multiple challenges in satisfying the hypothesis of Theorem~\ref{curd-from-design-thm}:  (i) finding the initial base blocks $B$ and $Y$, (ii) choosing the partition of ${V} - B$  as $A \cup  C$ and the partition of ${V} - Y$  as $X \cup  Z$, (iii) partitioning $A$ and $X$ into pairs and finding a bijection from $C$ to  $Z$ so that the three conditions of Theorem~\ref{curd-from-design-thm} are satisfied.   In Proposition~\ref{ten-verif-prop}, we meet all the conditions for $m=4$ and $n= 10$, and  likewise in Proposition~\ref{prop-26}  we meet all the conditions for 
 $m= 6$, $n = 26$ and  in Proposition~\ref{prop-82}, we meet all the conditions for 
 $m= 10$, $n = 82$.

 \begin{Prop}
 The blocks given in Table~\ref{ten-table-2} form a CURD with partition $4^12^3$.
 \label{ten-verif-prop}
 \end{Prop}
 
 \begin{proof}
 For the variety set ${V}  = \{0,1,2,3,4\}$, 
 the base blocks $B = \{1,2\}$ and $Y = \{0,3\}$ modulo $5$  generate a  cyclic $(10,5,4,2,1)$-BIBD.    The modulo $5$ difference sets are $D_1 = \{1,4\}$ and $D_2 = \{2,3\}$.   We partition ${V} - B$ as $A \cup C$ where $A = \{0,3\}$ and $C = \{4\}$, and we
 partition ${V} - Y$ as $X \cup Z$ where $X = \{1,2\}$ and $Z = \{4\}$.  
 
 Since $k=2$ we know $\binom{k}{2} = 1$ and our partition of $A$ is $\{0,3\}$.  The set $\{\pm(3-0) \pmod 5 \}$ is $\{2,3\}$ which equals $D_2$, so condition (1)  of Theorem~\ref{curd-from-design-thm}
 is satisfied.   Similarly, our partition of $X$ is $\{1,2\}$.  The set $\{\pm(2-1) \pmod 5 \}$ is $\{1,4\}$ which equals $D_1$, so condition (2) is satisfied.  
 
 Since $C = \{4\}$ we know $\ell = 1$ and we define $\phi:C \to Z$ by $\phi(4) = 4$.  Thus
 $\{(\phi(c_i) - c_i) \pmod 5: 1 \le i \le \ell\} = \{0\}$.  The original base blocks are $B = \{1,2\}$ and $Y = \{0,3\}$, so $b_1 = 1$, $b_2 = 2$, 
 $y_1 = 0$,  and $y_2 = 3$.  The set of differences $\{(y_j - b_i) \pmod 5: 1 \le i,j \le 2\}$ is $\{1,2,3,4\}$, thus condition (3) is satisfied.
 
 By Theorem~\ref{curd-from-design-thm}, the following base blocks generate at $4^1 \, 2^3$-CURD.
 \begin{itemize}
 \item Size $4$ base block:  $\{1,2,0',3'\}$
 \item Size $2$ base block within $S$: $\{0,3\}$
  \item Size $2$ base block within $S'$: $\{1',2'\}$
  \item Size $2$ base blocks between $S$ and $S'$:  $\{4,4'\}$
 \end{itemize}

 The base blocks appear in Week 1 of Table~\ref{ten-table-2}  and in subsequent weeks, the blocks are obtained from the previous week by mapping each element $i$ to $i \oplus 1$.
\end{proof}
 
   \begin{table}
 \centering
 \begin{tabular} {|l |l | } \hline
  Description& Base Blocks   \\ \hline
Block of size $6$ & $\{1,3,9,2',5',6'\}$    \\ \hline
Blocks within  $S$ &    $\{2,12\}$, \  $\{4,5\}$, \   \ $\{6,10\}$ \    \\ \hline
  Blocks within   $S'$ &   $\{4',11'\}$,    $\{7',12'\}$, \  $\{8',11'\}$    \\ \hline
  Blocks between $S$ and $S'$ & $\{0,0'\}$, \   \ $\{7,1'\}$, \  \ $\{8,3'\}$,  \  $\{11,9'\}$    \\ \hline
 \end{tabular}
 \caption{Base blocks for a $6^1 2^{10}$-CURD with $n=26$ and $\lambda = 1$.}
 \label{table-base-blocks-26} 
 \end{table}

 \begin{Prop}
 The  base blocks given in Table~\ref{table-base-blocks-26}  generate a CURD with partition $6^1 2^{10}$.
 \label{prop-26}
 \end{Prop}
 
 \begin{proof}
 For the variety set ${V}  = \{0,1,2, \ldots, 12\}$,  one can check that 
 the base blocks $B = \{1,3,9\}$ and $Y = \{2,5,6\}$ modulo $13$  generate a $(26, 13, 6,3,1)$-BIBD.    The modulo $13$ difference sets are $D_1 = \{2,5,6,7,8,11\}$ and $D_2 = \{1,3,4,9,10,12\}$.   We partition ${V} - B$ as $A \cup C$ where $A = \{2,4,5,6,10,12\}$ and $C = \{0,7,8,11\}$, and we
 partition ${V} - Y$ as $X \cup Z$ where $X = \{4,7,8,10,11,12\}$ and $Z = \{0,1,3,9\}$.  
 
 Since $k=3$ we know $\binom{k}{2} = 3$ and we choose to  partition  $A$ as $\{2,12\}  \cup \{4,5\} \cup \{6,10\}$.  The  set of residues modulo $13$ of the elements in   $\{\pm(12-2), \pm(5-4), \pm(10-6) \}$ is $\{1,3,4,9,10,12\}$ which equals $D_2$, so condition (1)  of Theorem~\ref{curd-from-design-thm}
 is satisfied.   Similarly, we partition $X$ as  $\{4,11\}  \cup \{7,12\} \cup \{8,10\}$.  The set of residues modulo $13$ of the elements in  $\{\pm(11-4), \pm(12-7), \pm(10-8) \}$ is $ \{2,5,6,7,8,11\}$ which equals $D_1$, so condition (2) is satisfied.  
 
 Define the bijection $\phi: C \to Z$ as $\phi(0) = 0$, $\phi(7) = 1$, $\phi(8) = 3$, and $\phi(11) = 9$,
  thus
 $\{(\phi(c_i) - c_i) \pmod {13}: 1 \le i \le 4\} = \{0,7,8, 11\}$.  The original base blocks are $B = \{1,3,9\}$ and $Y = \{2,5,6\}$, so $b_1 = 1$, $b_2 = 3$,  $b_3 = 9$, 
 $y_1 = 2$,  and $y_2 = 5$, and $y_3 = 6$.  The set of differences $\{(y_j - b_i) \pmod {13}: 1 \le i,j \le 3\}$ is $\{1,2,3,4,5,6,9,10,12\}$, thus condition (3) is satisfied.
 
 By Theorem~\ref{curd-from-design-thm}, the  base blocks in Table~\ref{table-base-blocks-26} generate a $6^1 \, 2^{10}$-CURD.
\end{proof}

\begin{table}
 \centering
 \begin{tabular} {|l |l | } \hline
  Description & Base Blocks   \\ \hline
  Size $10$ &   $\{1, 10, 16, 18, 37, 2', 20', 36', 32', 33' \} $  \\  \hline
Within $S$ &    $ \{5,6\}, \{9,19\}, \{39,14\}, \{8,26\}, \{21, 17\},   \{31, 38\}, \{23, 11\}, \{4,34\}, \{25, 28\}, \{40,12\} $   \\\hline
Within $S'$ &   $\{10',12'\}, \  \{18',38'\}, \  \{37', 28'\}, \  \{16',11'\}, \  \{1',34'\}, $  \\
&  $ \{21',35'\}, \  \{5',22'\},  \ \{8',27'\},  \ \{9',15'\}, \  \{39',24'\}$\\
  \hline
 Between  &  $\{2,30'\},  \, \{20,13'\}, \,\{36,7'\}, \, \{32,29'\}, \, \{33,3'\},  \, \{15, 23'\}, \, \{27,25'\}, \, \{24,4'\}$ \\ 
 $S$ and $S'$ &  $\{35,40'\}, \, \{22,31'\}, \, \{3,6'\},  \,\{30,19'\}, \, \{13,26'\}, \, \{7,14'\},  \, \{29,17'\},  \, \{0,0'\}$ \\ 
 \hline
 \end{tabular}
 \caption{Base blocks for a $10^1 2^{36}$-CURD with $n=82$ and $\lambda = 1$. }
 \label{table-base-blocks-82} 
 \end{table}

 \begin{Prop}
 The  base blocks given in Table~\ref{table-base-blocks-82}  generate a  CURD with partition $10^1 2^{36}$.
 \label{prop-82}
 \end{Prop}
 
 \begin{proof}

Let $V$ be the variety set   $ \{0,1,2, \ldots, 40\}$, and let $B = \{1, 10, 16, 18, 37\}$ and $Y = \{2,20, 32, 33, 36\}$.  One can check that in modulo $41$, the
  base blocks $B  $ and $Y $   generate a $(82, 41, 10,5 ,1)$-BIBD.    The  modulo $41$ difference sets   $D_1$ and $D_2$ are  listed below and these partition  $\{1,2,3, \ldots, 40\}$.

 \medskip
 \noindent
 $D_1 = \{\pm 2, \pm 5, \pm 6, \pm 8, \pm 9, \pm 14, \pm 15, \pm 17, \pm 19, \pm 20, \}   $ 

 \smallskip
  \noindent
 $D_2 = \{ \pm 1, \pm 3, \pm 4, \pm 7, \pm 10, \pm 11, \pm 12, \pm 13, \pm 16, \pm 18 \}.$
  \medskip
  
 We partition ${V} - B$ as $A \cup C$  and ${V} - Y$ as $X \cup Z$ where

  \medskip
 \noindent
 $A = \{4,5,6,8,9,11, 12, 14, 17, 19, 21, 23, 25, 26,  28, 31, 34, 38, 39, 40\} $
 
  \smallskip
  \noindent
 $  C = \{0,2,3,7, 13, 15, 20, 22, 24, 27, 29, 30, 32, 33, 35, 36 \}$

  \smallskip
 \noindent
 $X = \{1, 5, 8, 9, 10, 11, 12, 15, 16, 18,21, 22, 24, 27, 28, 34, 35, 37, 38, 39\}$

   \smallskip
 \noindent
 $Z = \{0,3,4, 6,7,13, 14, 17, 19, 23, 25, 26, 29, 30, 31, 40\}.$ 
 \medskip

\begin{table}
\begin{center}
 \begin{tabular}{|l  | l | l | l | l | l | l | l | l | l |l | }  \hline
$a_i$  & 4 & 5  & 8 & 9 &  11 & 12 & 14 & 17 &    25 & 31 \\ \hline
 $\hat{a}_i$& 34 & 6 & 26 & 19 &   23 & 40 & 39 & 21 & 28 &  38 \\  \hline
 $\pm(\hat{a}_i -a_i ) \pmod{41}$ & $\pm 11$ & $\pm 1$ &$\pm 18$ &$\pm 10$ &$\pm 12$ &$\pm 13$ &$\pm 16$ &$\pm 4$ &$\pm 3$ &$\pm 7$  \\ \hline
 \end{tabular}
 \end{center}
 
\medskip
\begin{center}
  \begin{tabular}{|l  | l | l | l | l | l | l | l | l | l |l | }  \hline
$x_i$ & 1 & 5  & 8 & 9 &  10 & 11 & 18 & 21 &    24 & 28 \\ \hline
 $\hat{x}_i$ & 34 & 22 & 27 & 15 &   12 & 16 & 38 & 35 & 39 &  37 \\  \hline
 $\pm(\hat{x}_i - x_i) \pmod{41}$ & $\pm 8$ & $\pm 17$ &$\pm 19$ &$\pm 6$ &$\pm 2$ &$\pm 5$ &$\pm 20$ &$\pm 14$ &$\pm 15$ &$\pm 9$  \\ \hline
 \end{tabular}
 \end{center} 
 
 \caption{The $20$ elements of $A$  (respectively $X$) partitioned into unordered pairs $\{a_i, \hat{a_i}\}$ 
 (respectively  $\{x_i, \hat{x_i}\}$) and the mod $41$ differences for each pair. }
 \label{partition-A-table}
 \end{table}
 Since $k=5$ we know $\binom{k}{2} = 10$ and our partition of $A$  as $\{ \{a_i,\hat{a}_i\}: 1 \le i \le 10 \} \}$ is given in  row 2 of Table~\ref{table-base-blocks-82}.  In  Table~\ref{partition-A-table},   the differences $\pm(\hat{a}_i - a_i) \pmod{41}$ are computed and this set of differences   equals $D_2$, so condition (1)  of Theorem~\ref{curd-from-design-thm}
 is satisfied. 
 Similarly, our partition of $X$ as $\{ \{x_i,\hat{x}_i\}: 1 \le i \le 10 \} \}$ is given in  row 3 of Table~\ref{table-base-blocks-82}.  In Table~\ref{partition-A-table}   the differences $\pm(\hat{x}_i-x_i) \pmod{41}$ are computed and this set of differences   equals $D_1$, so condition (2) is satisfied.

\begin{table}
\begin{center}
\begin{tabular} { |l | l | l | l | l | l | l | l | l | l | l | l | l | l | l | l | l | } \hline
 $c_j$ &  0 & 2 & 3 & 7 & 13 & 15 & 20 & 22 & 24 & 27 & 29 & 30 & 32 & 33 & 35 & 36 \\ \hline
 $\phi(c_j)$ &  0 & 30 & 6 & 14 & 26 & 23 & 13 & 31 & 4 & 25 & 17 & 19 & 29 & 3 & 40 & 7 \\ \hline
 diff &  0 & 28 & 3 & 7 & 13 & 8 & 34 & 9 & 21 & 39 & 29 & 30 & 38 & 11 & 5 & 12 \\ \hline
\end{tabular}
\end{center}
\caption{The mapping  $\phi: C \to Z$ and the differences $\phi(c_j) - c_j \pmod{41}$}
\label{table-C-mapping}
\end{table}

The bijection $\phi: C \to Z$ is given in Table~\ref{table-C-mapping} and the third row shows the quantities $\phi(c_j) - c_j \ \pmod {41}.$
\smallskip
  The original base blocks are $B = \{1, 10, 16, 18, 37\}$ and $Y = \{2, 20, 32, 33, 26\}$, so without loss of generality, $b_1 = 1$, $b_2 = 10$,  $b_3 = 16$, $b_4 = 18$, $b_5 = 37$
 and 
 $y_1 = 2$,   $y_2 = 20$, and $y_3 = 32$, $y_4 = 33$, $y_5 = 36$.  The set of differences $\{(y_j - b_i) \pmod {41}: 1 \le i,j \le 5\}$ is given in Table~\ref{yb-table}, and  this set of $25$ elements together with the $16$ entries in  the third row of Table~\ref{table-C-mapping} partition $V$.   Thus condition (3)  of Theorem~\ref{curd-from-design-thm}
 is satisfied.
 By Theorem~\ref{curd-from-design-thm}, the  base blocks in Table~\ref{table-base-blocks-82}
 generate at $10^1 \, 2^{36}$-CURD.
\end{proof}

 \medskip
\begin{table}
\begin{center}
 \begin{tabular} {c|ccccc}
 $y_j- b_i \pmod{41}$ &    1   &   10   & 16 & 18 & 37 \\ \hline
$2$   & 1 & 33 & 25 & 6 & 27  \\
$20$  & 19 & 10 & 2 & 24 & 4 \\
$32$ & 31 & 22 & 14 & 36 & 16 \\
$33$ & 32 & 23 & 15 & 37 & 17 \\
$36$ & 35 & 26 & 18 & 40 & 20 \\
\end{tabular}
\end{center}
\caption{The set of differences $\{(y_j - b_i) \pmod {41}: 1 \le i,j \le 5\}$ where $B = \{1, 10, 16, 18, 37\}$ and $Y = \{2,20, 32, 33, 36\}$ }
\label{yb-table}
\end{table}

Given the difficulty inherent in constructing CURDS of even moderate size, the reader may be interested to know how we were able to satisfy the conditions of Theorem~\ref{curd-from-design-thm} to produce the $10^12^{36}$-CURD in Proposition~\ref{prop-82}.  We will give a brief explanation of our process here following  the notation developed in Theorem~\ref{curd-from-design-thm}. 
For background on concepts and terminology from abstract algebra, see \cite{Hu03}.
Our construction starts with a judicious choice of base blocks for the initial $(82, 41, 10, 5 ,1)$-BIBD.  Identifying the variety set $V=\lbrace 0, 1, 2,\dots, 40\rbrace$ with the ring $\mathbb{Z}_{41}$, we recall that the set of non-zero elements of $V$ corresponds precisely to the set of invertible elements of $\mathbb{Z}_{41}$, and that this forms a cyclic group of order 40.  The block $B=\lbrace 1, 10, 18, 16, 37\rbrace$ is the (unique) 5-element subgroup of this group.  (Note: the order in which we have written the elements of $B$ is important; it corresponds to selecting 10 as the generator.)  The reader can verify that the set of differences, $D_1$, consists precisely of the union of four of the eight cosets of this subgroup (namely, $2B$, $5B$, $6B$, and $15B$).  The second block, $Y$, is obtained by multiplying $B$ by 2 (mod 41), and thus the set of differences, $D_2$, is also obtained by multiplying $D_1$ by 2 (mod 41).  This implies that $D_2$ is also a union of four cosets, and they turn out to be exactly the four we were missing.  Thus, $D_1$ and $D_2$ partition $\lbrace 1, 2, 3, \dots, 40\rbrace$.

The union of $B$ and $Y$ produces our block of size 10, as described in Theorem~\ref{curd-from-design-thm}.  In order to specify the pairs called for in Theorem~\ref{curd-from-design-thm}, we rely on the following observation.  Let's denote the elements of (the multiplicative subgroup) $B$ by $b_1, b_2, b_3, b_4, b_5$.  Let $\alpha B=\lbrace \alpha b_1, \alpha b_2, \alpha b_3, \alpha b_4, \alpha b_5\rbrace$ and $\beta B=\lbrace \beta b_1, \beta b_2, \beta b_3, \beta b_4, \beta b_5\rbrace$ be two cosets, for some choices of $\alpha$ and $\beta$.  Then the set of corresponding differences, $\lbrace (\alpha-\beta)b_1, (\alpha-\beta)b_2, (\alpha-\beta)b_3, (\alpha-\beta)b_4, (\alpha-\beta)b_5\rbrace$ is {\em also} a coset of $B$ (generated by $\alpha-\beta$).  This observation allows us to construct the sets $A$, $X$, $C$ and $Z$ as unions of cosets.  For example, the set $A$ consists of the cosets $5B=\lbrace 5, 9, 8, 39, 21\rbrace$, $6B=\lbrace 6, 19, 26, 14, 17\rbrace$, $31B=\lbrace 31, 23, 25, 4, 40\rbrace$ and $38B=\lbrace 38, 11, 28, 34, 12\rbrace$.  We paired $5B$ with $6B$, forming the desired pairs by pairing corresponding elements.   The differences given by these five pairs comprise exactly the cosets $B$ and $40B$.  In showing this example, we emphasize again the importance of keeping the elements of each coset consistent with the order we imposed on $B$, as this is what ensures that when we compute the differences of corresponding elements, we again obtain cosets. Similarly, we paired $31B$ with $38B$; the corresponding differences being exactly the cosets $7B$ and $34B$.  The reader can verify that this is exactly what is listed (albeit not organized in this exact way) in the first part of Table~\ref{partition-A-table}.

Continuing to reason in the above fashion, we were able to similarly construct the sets $C$, $X$ and $Z$ in such a way that we satisfied the hypotheses of Theorem~\ref{curd-from-design-thm}.  To be clear, we do not claim that our particular choices of how to match the cosets was either inevitable or unique, merely that by being able to reduce the problem from one concerning 80 elements to one concerning 16 sets, we were able to decrease the complexity enough to proceed by hand.  Looking back at the $6^12^{10}$-CURD constructed in Proposition~\ref{prop-26}, we see that similar reasoning can be applied in that case as well, and we wonder whether the idea expressed here can be extended to construct a larger family of CURDs.

\section{Constructing a CURD where $m$ is a power of an odd prime}

In this section, we will  construct a CURD with partition $m^12^{\frac{m^2-m}{2}}$ in the case where $m=p^f$ for any odd prime $p$ and any positive integer $f$.  We will build such a CURD by first providing a construction of an $m^m$-CURD, and then adapting that to obtain the desired CURD.  In order to fix ideas, we will first give the proof in the case in which $m$ is  an odd prime, before giving the general proof.

\begin{Def} \rm
   For an odd prime $p$, the  design $ {\mathcal D}_p =  ({V}_{p}, {\mathcal 
    B}_{p})$  
    is defined as follows:    $V_{p} = \{ (x,y): 1 \le x,y \le p\}$ and   ${\mathcal 
    B}_{p}$ consists of  blocks $C^r_i $ for $1 \le i,r \le p$ and  blocks 
$C^{\infty}_i$ for $1 \le i \le p$ where
\smallskip

    $C^r_i = \{(x,y) \in V_p:  y \equiv rx+i \pmod p\}, {\rm\ and}$

      $C^{\infty}_i = \{(x,y) \in V_p:   x \equiv i \pmod p\}.$

   \label{prime-design-def}   
\end{Def}

Table~\ref{twenty-five-table} shows the 25 varieties in the  design  ${\mathcal D}_5$ arranged in rows and columns.  The blocks  $C^{\infty}_i $ appear as  the set of varieties in the same vertical line and blocks  $C^r_i $ appear as the set of varieties in a line of slope $r$  (in modulo $5$), for example, ${C}^{1}_1  = 
\{ (1,2), (2,3), (3,4), (4,5), (5,1)\}$.

\begin{table}
\begin{center}
\begin{tabular}{| c ||  c | c | c | c | c | } \hline
& $ C^{\infty}_1$ & $ C^{\infty}_2$ & $ C^{\infty}_3$ & $ C^{\infty}_4$ & $ C^{\infty}_5$ \\ \hline \hline
$ C^5_5$ & $(1,5)$ & $(2,5)$  &  $(3,5)$  &  $(4,5)$  &  $(5,5)$  \\ \hline
$ C^5_4$ & $(1,4)$ & $(2,4)$  &  $(3,4)$  &  $(4,4)$  &  $(5,4)$  \\ \hline
$ C^5_3$ & $(1,3)$ & $(2,3)$  &  $(3,3)$  &  $(4,3)$  &  $(5,3)$  \\ \hline
$ C^5_2$ & $(1,2)$ & $(2,2)$  &  $(3,2)$  &  $(4,2)$  &  $(5,2)$  \\ \hline
$ C^5_1$ & $(1,1)$ & $(2,1)$  &  $(3,1)$  &  $(4,1)$  &  $(5,1)$  \\ \hline
 \end{tabular}

 \end{center}
 \caption{The $25$ varieties in the design ${\mathcal D}_5$ arranged in rows and columns.}
\label{twenty-five-table}
\end{table}

\begin{Prop}
    If $p$ is an odd prime, then the design ${\mathcal D}_p$ is a $p^p$-CURD when the blocks are partitioned into the following $p+1$ parallel classes:  for $1 \le r \le p$, the blocks  in week $r$ are $C^r_1, C^r_2 , \ldots, C^r_p $    and  the blocks in week $p+1$ are $C^{\infty}_1, C^{\infty}_2, \ldots, C^{\infty}_p $.
    \label{planelemma}
\end{Prop}

\begin{proof}
    Given the way that we have defined $\mathcal{D}_{p}$, it will be helpful to think geometrically.  For any $r$ (including $\infty$), the collection ${C}^r_1, {C}^r_2 , \ldots, {C}^r_p $ represents the set of lines of slope $r$.  We claim first that the design ${\mathcal D}_p$ is pairwise balanced (with $\lambda=1$).  Since each variety can be thought of as a point in the plane, and the blocks represent all of the lines, the fact that each pair of varieties appears together in a block exactly once is exactly the fact that any two distinct points determine a unique line. To show that $\mathcal{D}_p$ is resolvable, note that for each $r$ (including $\infty$), the collection ${C}^r_1, {C}^r_2 , \ldots, {C}^r_p$ is a parallel class.  Finally, since each parallel class is a partition of $V_p$ into $p$ sets of size $p$, we see that $\mathcal{D}_p$ is a $p^p$-CURD.
\end{proof}

Since $p$ is odd in the definition of ${\mathcal D}_p$, there are an even number of parallel classes in $\mathcal{D}_p$ and they can be considered in pairs.  This motivates the next definition.

\begin{Def} \rm
    Let $q$ be an odd integer and $S$ a set with $|S| = q^2$.  A \emph{double $q$-partition} of $S$ consists of two partitions $\{A_1, A_2, \cdots , A_q \}$  and $\{B_1, B_2, \cdots , B_q \}$ 
    of $S$    so that $|A_i| = |B_i| = q$ for $1 \le i \le q$ and $|A_i \cap B_j| = 1$ for $1 \le i,j \le q$.
    \label{double-def}
\end{Def}

%\begin{example} \rm
    As an example, let $q=5$ and $S = \{(a,b): 1 \le a,b \le 5 \}$.  For $1 \le i \le 5$, let $A_i =  \{(a,i): 1 \le a \le 5\}$ and 
    $B_i = \{(i,b): 1 \le b \le 5\}$.  Then  
    $A_1 \cup A_2 \cup \cdots \cup A_5$ and $B_1 \cup B_2 \cup \cdots \cup B_5$ form a double $5$-partition of $S$.  These correspond to the rows and columns of the entries in Table~\ref{twenty-five-table} where   $A_i = C_i^5$ and $B_i = C_i^{\infty}$.
%\end{example}

In transforming $\mathcal{D}_p$ to a 
 CURD with partition $p^12^{\frac{p^2-p}{2}}$, we will convert a double $q$-partition of a set $S$ into a set of $q$ partitions of $S$, each of which contains one part of size $q$ and the remaining parts of size $2$.  

\begin{Def}
\rm
      Let $q$ be an odd integer and $S$ a set with $|S| = q^2$.  A  \emph{$(q, q^1 2^{\frac{q^2-q}{2}})$-partition} of $S$ consists of $q$ partitions 
      ${\mathcal P}_1, {\mathcal P}_2, \ldots, {\mathcal P}_q$ where each ${\mathcal P}_i$ is a partition of $S$ consisting of one part of size $q$ and $\frac{(q-1)q}{2}$ parts of size 2.
\end{Def}

To achieve this transformation, we use a result about edge coloring from graph theory.
As discussed in Section~\ref{sec-necessary-conditions}, a CURD with $n$ varieties and $\lambda = 1$ can be viewed as a partition of the edges of the complete graph $K_n$ into color classes. We are interested in the case in which each color class is a matching, that is, it consists of a set of edges, no two of which share an endpoint.
For a graph $G$, the \emph{edge chromatic number} $\chi'(G)$ is the minimum number of colors needed to color the edge set of $G$ so that incident edges get different colors.  The following result about the edge chromatic number of a complete graph on an even number of vertices appears in many graph theory texts such as \cite{We01}.
\begin{Thm}
    If $n$ is an integer, then $\chi'(K_{2n}) = 2n-1$.
    \label{edge-chrom-thm}
\end{Thm}

A standard proof of Theorem~\ref{edge-chrom-thm} is constructive and uses the following coloring.   If the vertex set is  $\{x,  1, 2, 3, \ldots, 2n-1\}$,  then  the edges that receive color $1$ are  $\{1,x\}, \{2, 2n-1\}, \{3, 2n-2\}, \ldots, \{n, n+1\} $ and subsequent color classes are 
determined by leaving $x$ unchanged and adding $1$ to each remaining number in modulo $2n-1$.   Note that each color class is a partition of the $2n$ vertices into $n$ pairs. The first two color classes are shown in each graph $G_j$ in Figure~\ref{fig-Gj} where $n=3$ and the vertices are labeled $x, (j,1), (j,2), (j,3), (j,5), (j,5)$ instead of $x, 1,2,3,4,5$.

If $q$ is an odd integer and $S$ is a set with $|S| = q^2$, 
the  next lemma shows how to convert a double $q$-partition of $S$
into  a $(q, q^1 2^{\frac{q^2-q}{2}})$-partition  of a set $S$.  The construction involves  edge colorings of a sequence of complete graphs $G_j$ and is illustrated in Figure~\ref{fig-Gj}.

\begin{lemma}\label{double-partition}
  Let $q$ be an odd integer and $S$ a set with $|S| = q^2$.  If  $\{A_1, A_2, \ldots, A_q\}$ and  $\{B_1, B_2, \ldots, B_q\}$ form a double $q$-partition of $S$ then there exists a $(q, q^1 2^{\frac{q^2-q}{2}})$-partition $    {\mathcal P}_1, {\mathcal P}_2, \ldots, {\mathcal P}_q$ of $S$ so that if $u,v \in S$  then $u$ and $v$ appear together in a set in the list 
  $A_1, A_2, \ldots, A_q, B_1, B_2, \ldots, B_q$ if and only if they appear together in  a set in the list ${\mathcal P}_1, {\mathcal P}_2, \ldots, {\mathcal P}_q$.
    
\end{lemma}
   
    \begin{proof}
        By Definition~\ref{double-def} we know $|A_i \cap B_j| = 1$ for $1 \le i,j \le q$, so we can let $v_{ij}$ be the vertex in $A_i \cap B_j$.  Further, let $G_j$ be the complete graph on $q+1$ vertices with vertex set $B_j \cup \{x\}$.  Using Theorem~\ref{edge-chrom-thm} we know $\chi'(G_j) = q$.  Color the edges of each $G_j$ properly using $q$ colors so that the edge between $x$ and $v_{ij}$ receives color $i$.  
        Each vertex of $G_j$ has degree $q$, and so it is incident to an edge of each of the $q$ colors.  There are  $q+1$ vertices in $G_j$, so there are $\frac{q+1}{2}$ edges of each color in our coloring of $G_j$.

        Let partition ${\mathcal P}_i$ consist of the size $q$ set $A_i$ and the pairs of vertices in $B_j - \{v_{ij} \}$ corresponding to the color $i$ edges of $G_j$ for $1 \le j \le q$.  Since we are not including the color $i$ edge between $x$ and $v_{ij}$, there are  $\frac{q-1}{2}$ color $i$ edges remaining  in each of $G_1, G_2, \ldots, G_q$ and hence
         $\frac{(q-1)q}{2}$ parts of size 2 in ${\mathcal P}_i$ for each $i$. 

         We next show that ${\mathcal P}_i$ is a partition of $S$ for each $i$, that is, each element $v$ of $S$ is in a unique part of ${\mathcal P}_i$.  If $v \in A_i$ then it appears in the size $q$ part of ${\mathcal P}_i$, and by definition, $v = v_{ij}$ for each $j$, so $v$ is not included in any other  size 2 part of ${\mathcal P}_i$. If $v \not\in A_i$ then by Definition~\ref{double-def}, there exists a unique $j$ for which $v \in B_j$ and by our choice of $v_{ij}$ we know $v \neq v_{ij}$.  Then there is exactly one color $i$ edge incident to $v$ in $G_j$, and hence $v$ appears in exactly one part of size 2 in ${\mathcal P}_i$.

         Finally, we verify the last assertion in the theorem.  Suppose $u,v \in S$ and they appear together in a set in the list  
  $A_1, A_2, \ldots, A_q, B_1, B_2, \ldots, B_q$.  If $u,v \in A_i$ for some $i$, then they appear together in partition ${\mathcal P}_i$.  Otherwise, $u,v \in B_j$ for some $j$, so $u$ and $v$ are vertices in the graph $G_j$, and neither is equal to $v_{ij}$.      The edge $uv$  has some color $k$ in $G_j$ and thus $u$ and $v$ appear in a part of size 2 in the partition ${\mathcal P}_k$.  Thus $u$ and $v$ appear together in a set in the list  ${\mathcal P}_1, {\mathcal P}_2, \ldots, {\mathcal P}_q$.

  A straightforward counting argument shows that no pair of vertices can appear twice in a set in the list  ${\mathcal P}_1, {\mathcal P}_2, \ldots, {\mathcal P}_q$ because there are $\frac{q^2(q-1)}{2}$ pairs of elements appearing together in  one of $B_1, B_2, \ldots, B_q$, and the same number appearing together in ${\mathcal P}_1, {\mathcal P}_2, \ldots, {\mathcal P}_q$.
           \end{proof}

Figure~\ref{fig-Gj} shows the vertex sets for the graphs $G_j$ in the case $q=5$ where the $A_i$ represent the horizontal lines and the $B_j$ represent vertical lines in Table~\ref{twenty-five-table}.  Following the proof of Theorem~\ref{edge-chrom-thm}, the edges for color 1 are shown in red and the edges for color 2 are shown in blue.  The blue edges are obtained geometrically from the red edges by rotating the figure and the subsequent color classes can be obtained by additional rotations.   Thus the partition ${\mathcal P}_1$ consists of the $5$-set $\{(1,1), (2,1), (3,1), (4,1), (5,1)\}$ (shown in Figure~\ref{fig-Gj} as the vertices paired with $x$) and the ten  $2$-sets  $\{(1,2)(1,5)\}, \{(1,3)(1,4)\}, \ldots, \{(5,3)(5,4)\}$
corresponding to the endpoints of the solid red edges in Figure~\ref{fig-Gj}.

\begin{figure} 

\newgeometry{left=0cm, right=0cm}

\begin{tikzpicture}[xscale=0.675, yscale=1.17]
\tikzstyle{vertex}=[circle,fill=black!25,minimum size=10pt,inner sep=3pt]
\begin{scope}[shift={(0,0)}]
\node[vertex,black][label=below:$x$] (0) at (1,1.2){};
\node[vertex,black][label=below:${(1,4)}$] (1) at (0,0.2) {}; 
\node[vertex,black][label=below:${(1,3)}$] (2) at (2,0.2) {}; 
\node[vertex,black][label=above:${(1,2)}$] (3) at (2.4,1.74) {}; 
\node[vertex,black][label=above:${(1,1)}$] (4) at (1,2.8) {}; 
\node[vertex,black][label=above:${(1,5)}$] (5) at (-0.4,1.74) {}; 
\draw[<->,shorten <=2pt][blue,thick] (2) -- (4) -- cycle;
\draw[<->,shorten <=2pt][blue,thick] (1) -- (5) -- cycle;
\node at (1.2,-1) {$j=1$};
\draw[<->,shorten <=2pt][red,thick] (3) -- (5) -- cycle;
\draw[<->,shorten <=2pt][red,thick] (1) -- (2) -- cycle;
\draw[<->,shorten <=2pt][red,thick, dashed] (0) -- (4) -- cycle;
\draw[<->,shorten <=2pt][blue, thick,dashed] (0) -- (3) -- cycle;
\end{scope}

\begin{scope}[shift={(4.5,0)}]
\node[vertex,black][label=below:$x$] (0) at (1,1.2){};
\node[vertex,black][label=below:${(2,4)}$] (1) at (0,0.2) {}; 
\node[vertex,black][label=below:${(2,3)}$] (2) at (2,0.2) {}; 
\node[vertex,black][label=above:${(2,2)}$] (3) at (2.4,1.74) {}; 
\node[vertex,black][label=above:${(2,1)}$] (4) at (1,2.8) {}; 
\node[vertex,black][label=above:${(2,5)}$] (5) at (-0.4,1.74) {}; 
\draw[<->,shorten <=2pt][blue,thick] (2) -- (4) -- cycle;
\draw[<->,shorten <=2pt][blue,thick] (1) -- (5) -- cycle;
\draw[<->,shorten <=2pt][red,thick] (3) -- (5) -- cycle;
\draw[<->,shorten <=2pt][red,thick] (1) -- (2) -- cycle;
\draw[<->,shorten <=2pt][red, thick,dashed] (0) -- (4) -- cycle;
\draw[<->,shorten <=2pt][blue,thick, dashed] (0) -- (3) -- cycle;
\node at (1.2,-1) {$j=2$};
\end{scope}

\begin{scope}[shift={(9,0)}]
\node[vertex,black][label=below:$x$] (0) at (1,1.2){};
\node[vertex,black][label=below:${(3,4)}$] (1) at (0,0.2) {}; 
\node[vertex,black][label=below:${(3,3)}$] (2) at (2,0.2) {}; 
\node[vertex,black][label=above:${(3,2)}$] (3) at (2.4,1.74) {}; 
\node[vertex,black][label=above:${(3,1)}$] (4) at (1,2.8) {}; 
\node[vertex,black][label=above:${(3,5)}$] (5) at (-0.4,1.74) {}; 
\draw[<->,shorten <=2pt][blue,thick] (2) -- (4) -- cycle;
\draw[<->,shorten <=2pt][blue,thick] (1) -- (5) -- cycle;
\draw[<->,shorten <=2pt][red,thick] (3) -- (5) -- cycle;
\draw[<->,shorten <=2pt][red,thick] (1) -- (2) -- cycle;
\draw[<->,shorten <=2pt][red, dashed,thick] (0) -- (4) -- cycle;
\draw[<->,shorten <=2pt][blue, dashed,thick] (0) -- (3) -- cycle;
\node at (1.2,-1) {$j=3$};
\end{scope}

\begin{scope}[shift={(13.5,0)}]
\node[vertex,black][label=below:$x$] (0) at (1,1.2){};
\node[vertex,black][label=below:${(4,4)}$] (1) at (0,0.2) {}; 
\node[vertex,black][label=below:${(4,3)}$] (2) at (2,0.2) {}; 
\node[vertex,black][label=above:${(4,2)}$] (3) at (2.4,1.74) {}; 
\node[vertex,black][label=above:${(4,1)}$] (4) at (1,2.8) {}; 
\node[vertex,black][label=above:${(4,5)}$] (5) at (-0.4,1.74) {}; 
\draw[<->,shorten <=2pt][blue,thick] (2) -- (4) -- cycle;
\draw[<->,shorten <=2pt][blue,thick] (1) -- (5) -- cycle;
\draw[<->,shorten <=2pt][red,thick] (3) -- (5) -- cycle;
\draw[<->,shorten <=2pt][red,thick] (1) -- (2) -- cycle;
\draw[<->,shorten <=2pt][red, dashed,thick] (0) -- (4) -- cycle;
\draw[<->,shorten <=2pt][blue, dashed,thick] (0) -- (3) -- cycle;
\node at (1.2,-1) {$j=4$};
\end{scope}

\begin{scope}[shift={(18,0)}]
\node[vertex,black][label=below:$x$] (0) at (1,1.2){};
\node[vertex,black][label=below:${(5,4)}$] (1) at (0,0.2) {}; 
\node[vertex,black][label=below:${(5,3)}$] (2) at (2,0.2) {}; 
\node[vertex,black][label=above:${(5,2)}$] (3) at (2.4,1.74) {}; 
\node[vertex,black][label=above:${(5,1)}$] (4) at (1,2.8) {}; 
\node[vertex,black][label=above:${(5,5)}$] (5) at (-0.4,1.74) {}; 
\draw[<->,shorten <=2pt][blue,thick] (2) -- (4) -- cycle;
\draw[<->,shorten <=2pt][blue,thick] (1) -- (5) -- cycle;
\draw[<->,shorten <=2pt][red,thick] (3) -- (5) -- cycle;
\draw[<->,shorten <=2pt][red,thick] (1) -- (2) -- cycle;
\draw[<->,shorten <=2pt][red, dashed,thick] (0) -- (4) -- cycle;
\draw[<->,shorten <=2pt][blue, dashed,thick] (0) -- (3) -- cycle;
\node at (1.2,-1) {$j=5$};
\end{scope}
\end{tikzpicture}
\restoregeometry

\caption{The vertex sets for graphs $G_j$ when $q=5$ with the edges shown for color 1 in red and color 2 in blue.}
\label{fig-Gj}
\end{figure}
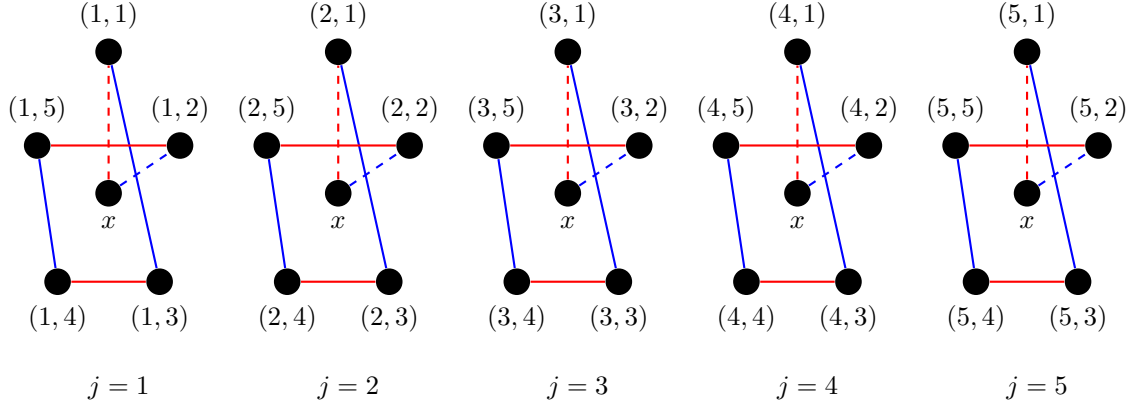

We now have the tools to transform the $p^p$-CURD ${\mathcal D}_p$ into 
a  CURD with partition $p^12^{\frac{p^2-p}{2}}$.

\begin{Thm}\label{primecurd}
If $p$ is an odd prime then we  can construct a  CURD with partition $p^12^{\frac{p^2-p}{2}}$.
\end{Thm}

\begin{proof}
We begin with the CURD $\mathcal{D}_p$. Pair off the parallel classes into $\frac{p+1}{2}$ pairs as follows.  Define 
$$\mathcal{Y}_0=\lbrace C_1^0,\dots, C_p^0\} \cup \{ C_1^\infty,\dots, C_p^\infty\rbrace$$
and for each $d$: $1\le d\le\frac{p-1}{2}$ define
$$\mathcal{Y}_d=\lbrace C_1^d,\dots, C_p^d\} \cup \{ C_1^{p-d},\dots, C_p^{p-d}\rbrace$$
For each $d$ ($0\le d\le\frac{p-1}{2}$), note that $\mathcal{Y}_d$ is a double $p$-partition by Proposition~\ref{planelemma}.
Replace $\mathcal{Y}_d$ with the $(p, p^1 2^{\frac{p^2-p}{2}})$-partition constructed in Lemma~\ref{double-partition}, and denote  it by $\mathcal{P}^d_1, \mathcal{P}^d_2, \ldots, \mathcal{P}^d_p$.  We define 
design $\mathcal{D}'_p$ where the variety set is ${V}_p$ and the set of blocks is 
 $\mathcal{B}'_p$,  defined by
$$\mathcal{B}'_p=\bigcup_{\substack{1\le i\le p\\0\le d\le\frac{p-1}{2}}}\mathcal{P}_i^d.$$

The verification that $\mathcal{D}'_p$ is a CURD relies heavily on the fact that $\mathcal{D}_p$ is a CURD, together with Lemma~\ref{double-partition}.  By Proposition~\ref{planelemma}, any two varieties appear together in exactly one block of $\mathcal{D}_p$.  By Lemma~\ref{double-partition}, and the construction of $\mathcal{D}'_p$, two varieties appear together in a block in $\mathcal{D}'_p$ if and only if they appear together in a block in $\mathcal{D}_p$.  Thus, $\mathcal{D}'_p$ is pairwise balanced (with $\lambda=1$).  Further relying on Lemma~\ref{double-partition}, since each $\mathcal{P}_i^d$ is a partition of $V_p$, we see that $\mathcal{D}'_p$ is resolvable where the  $\mathcal{P}_i^d$ are the parallel classes.  Finally, since each partition $\mathcal{P}_i^d$ consists of one set of size $p$ and $\frac{p^2-p}{2}$ sets of size 2, we see that $\mathcal{D}'_p$ is a CURD with the desired parameters.
\end{proof}

By adapting the arguments of this section to finite fields, we can obtain a stronger version of Theorem~\ref{primecurd}.  Since the core argument is essentially the same, we will give abbreviated proofs where appropriate.

If $p$ is an odd prime and $f$ is a positive integer, then it is well known that there is a (unique) finite field of size $q=p^f$, which we will call $F_q$.  By replacing arithmetic modulo $p$ with arithmetic in $F_q$, we can proceed as before.

\begin{Def} \rm 
   Let $p$ be an odd prime, let $f$ be a positive integer, and set $q=p^f$. Define the  design $ {\mathcal D}_q =  ({V}_{q}, {\mathcal 
    B}_{q})$  as follows:    $V_q = \{ (x,y): x,y \in F_q\}$ and   ${\mathcal 
    B}_{q}$ consists of  blocks $C^r_i $ for $i,r \in F_q$ and  blocks 
$C^{\infty}_i$ for $i\in F_q$ where
\smallskip

    $C^r_i = \{(x,y): x,y \in F_q$ and $y = rx+i\}, {\rm\ and}$

      $C^{\infty}_i = \{(x,y):  x,y \in F_q$ and $x = i \}.$

   \label{primepower-design-def}   
\end{Def}

The proof that ${\mathcal D}_q$ is a CURD is completely analogous to the proof of Proposition~\ref{planelemma}.

\begin{Prop}
    If $p$ is an odd prime, $f$ is a positive integer, and $q=p^f$ then the design ${\mathcal D}_q$ is a $q^q$-CURD where parallel class $r$ consists of the blocks $C^r_1, C^r_2 , \ldots, C^r_p $, and $r$ ranges over all of the elements of $F_q$, including the symbol $\infty$.
    \label{planepowerlemma}
\end{Prop}

\begin{proof}
    As in the proof of Proposition~\ref{planelemma}, we think of $V_q$ as a plane (a 2-dimensional vector space over $F_q$), and the blocks $\mathcal{B}_q$ as the set of lines in this plane.  Under this identification, the proof is exactly analogous to the proof of Proposition~\ref{planelemma}.
\end{proof}

A construction of a  CURD with partition $q^12^{\frac{q^2-q}{2}}$ (for $q$ a power of an odd prime) also follows analogously.

\begin{Thm}\label{primepowercurd}
    If $p$ is an odd prime,  $f$ is a positive integer, and  $q=p^f$ then we   can construct a  CURD with partition $q^12^{\frac{q^2-q}{2}}$.
\end{Thm}

\begin{proof}
We begin with the CURD $\mathcal{D}_q$. As in the proof of Theorem~\ref{primecurd}, pair off the parallel classes into $\frac{q+1}{2}$ pairs, and replace each pair with a $(q, q^1 2^{\frac{q^2-q}{2}})$-partition.  The design $\mathcal{D}'_q$, whose varieties are $V_q$ and whose blocks are the union of the $\frac{q+1}{2}$ $(q, q^1 2^{\frac{q^2-q}{2}})$-partitions just constructed, is shown to be a CURD exactly as in the proof of Theorem~\ref{primecurd}.
\end{proof}

In the final result of this paper, we construct a CURD with $\lambda = 2$ that satisfies the additional requirement that each pair of varieties appears together in a block of each size. 

\begin{Cor}\label{primepowerCor}
    If $p$ is an odd prime,  $f$ is a positive integer, and  $q=p^f$ then there exists a $q^12^{\frac{q^2-q}{2}}$-CURD in which $\lambda = 2$ and each pair of varieties appears together once in an $m$-block and once in a $2$-block.
\end{Cor}

\begin{proof}
    In proving Theorems~\ref{primecurd} and \ref{primepowercurd}, we rely on Lemma~\ref{double-partition} which converts a double $q$-partition into a $(q,q^12^{\frac{q^2-q}{2}})$-partition. In the new partition, varieties in any set $A_i$ remain in a block of size $q$ while varieties in any set $B_j$ end up in a block of size $2$.  Now for each pair of parallel classes in the double $q$-partition,  apply Lemma~\ref{double-partition}  a second time, reversing the roles of the $A$'s and $B$'s.  This constructs the desired CURD.
\end{proof}

\section{Conclusion}
In this paper we investigate CURDS in which each parallel class consists of one block of size $m$ (for $m \ge 3$) and the remaining blocks of size $2$.  Our study was motivated by the question of dividing a class of students into groups each week where there is a  group of size $m$ working with the teacher and the remaining students work in pairs.  For two students, the experience of working together as a pair is qualitatively different from being together in the size $m$ group, and yet the CURD formulation treats them as equivalent.   In Corollary~\ref{primepowerCor}, we build a CURD with the additional feature that each pair of students be together once in a group of size $m$ and once as a pair.  More generally, it would be interesting to determine which CURDS can be modified so that each pair of varieties appears together in a block of each size.

\medskip

\noindent
{\bf Acknowledgement:}  The authors are grateful to Douglas Stinson for making us aware of the work of Danziger and Stevens.

\end{document}